\documentclass[11pt]{amsart}
\usepackage{amssymb,amsmath,amsthm}
\newtheorem{theorem}{Theorem}
\newtheorem{proposition}[theorem]{Proposition}

\newtheorem{lemma}[theorem]{Lemma}

\begin{document}

\title[Inverse mean curvature flow in hyperbolic space]{Inverse mean curvature flows in the hyperbolic 3-space revisited}
\author{Pei-Ken Hung and Mu-Tao Wang}
\begin{abstract}
This note revisits the inverse mean curvature flow in the 3-dimensional hyperbolic space. In particular, we show that the limiting shape is not necessarily round after scaling, thus resolving an inconsistency in the literature.  
\end{abstract}

\address{Department of Mathematics \\ Columbia University \\ 2990 Broadway \\ New York, NY 10027}
\thanks{ The second author was supported by the National Science Foundation under grant DMS-1105483 and DMS-1405152. The authors would like to thank Andre Neves for his interest in this work. }

\maketitle

\section{Introduction}
 Let $(\mathbb{H}^3,\bar{g})$ be the 3-dimensional hyperbolic space with the metric \[\bar{g}=dr^2+(\sinh^2 r)\sigma\] in the $(r, \theta)$ coordinates, where $\sigma=\sigma_{ij}d\theta^i d\theta^j$ is the standard metric on $S^2$. We show that there exists a star-shaped mean convex 2-surface $\Sigma_0$ in  $\mathbb{H}^3$ such that the inverse mean curvature flow with  $\Sigma_0$ as the initial surface does not converge to a round sphere (of constant curvature) after scaling. Such an example on an asymptotic hyperbolic 3-space with positive mass was constructed
 by Neves \cite{Neves} to demonstrate the impossibility of proving the hyperbolic Penrose inequality by the method of the inverse mean curvature
 flow. However, in \cite[Theorem 1.2 and 6.11]{Gerhardt2} it was claimed that the inverse mean curvature behaves better on the hyperbolic 3-space and deforms the induced metric on the surface to a round one after scaling. We show by a concrete example in this note that the claim does not hold true.  More precisely, we prove:
 \begin{theorem}
 There exists a star-shaped mean convex closed surface $\Sigma_0$ in $\mathbb{H}^3$  that has the following property. Let $\Sigma_t$ be the inverse mean curvature
 flow of $\Sigma_0$, and $|\Sigma_t|$ and $g_t$ be the area of $\Sigma_t$ and the induced metric on $\Sigma_t$, respectively. As $t\rightarrow \infty$, $|\Sigma_t|^{-1}g_t$ converges to a metric on $S^2$ that is not of constant curvature. 
 \end{theorem}
 
 The construction in \cite{Neves} relies on the positivity of the limit of the Hawking mass. However, in the hyperbolic case, the Hawking mass always limits
 to zero along the inverse mean curvature flow. We introduce a new geometric quantity (modified Hawking mass) to handle this degenerate phenomenon. 
 A new monotonicity formula for the inverse mean curvature flow is discovered along the way. 
 
The result in this paper affirms that the limiting behavior of the inverse mean curvature flow 
depends distinctively on the structure of the end at infinity. For an (asymptotically) Euclidean end, the limiting shape is always round and this fact plays a critical role in the inverse mean curvature flow proof of both the Riemannian Penrose inequality \cite{Huisken-Ilmanen} and quermassintegral inequalities \cite{Guan-Li}. In the asymptotically hyperbolic case, it is already demonstrated by Neves \cite{Neves} that the limiting shape is not necessarily round. We show that even in the hyperbolic case when the Hawking mass degenerates, this phenomenon persists. A new strategy relying on the conformal structure and the Sobolev type inequality
on the sphere at infinity was devised in \cite{Brendle-Hung-Wang, Brendle-Wang}  to tackle this difficulty and to obtain Penrose-Gibbon inequalities in the asymptotically flat case. Such a strategy has been further developed by several authors \cite{Ge-Wang-Wu, Lopes-de-Lima-Girao} to solve related problems.

\section{A modified Hawking mass and its monotonicity along the inverse mean curvature flow}
We recall the Hawking mass for a closed embedded surface $\Sigma$ in $\mathbb{H}^3$:
\[m_{H}(\Sigma)=\sqrt{\frac{|\Sigma|}{16\pi}}\left(1-\frac{1}{16\pi}\int_\Sigma (H^2-4)d\mu\right).\]
By the Gauss equation in $\mathbb{H}^3$, we have $H^2-4=4K+2|\mathring{A}|^2$, where $K$ is the Gauss curvature of $\Sigma$ and $\mathring{A}$ is the traceless part of the second fundamental form $A$ with $|A|^2=|\mathring{A}|^2+\frac{1}{2} H^2$. Therefore, one rewrites  
\begin{equation}\label{traceless}
 1-\frac{1}{16\pi}\int_\Sigma (H^2-4)d\mu=-\frac{1}{8\pi}\int_\Sigma |\mathring{A}|^2d\mu.
 \end{equation}
 
 In this article, we consider a modified quantity
 \begin{equation}
 \tilde{m}(\Sigma)=-|\Sigma| \int_\Sigma |\mathring{A}|^2d\mu.
 \end{equation}
 Though the scaling of $\tilde{m}(\Sigma)$ is no longer the same as a mass, we call it the modified Hawking mass in this note.

Consider the inverse mean curvature flow $\Sigma_t$ of a closed embedded surface $\Sigma_0$, which deforms the surface in the normal direction with speed $\frac{1}{H}$. In the following, we compute the evolution equation of the modified Hawking mass $\tilde{m}(\Sigma_t)$. The evolution equation for mean curvature is
\[  \frac{\partial H}{\partial t}=\frac{\Delta H}{H^2}-2\frac{|\nabla H|^2}{H^3}-\frac{|A|^2}{H}+\frac{2}{H}. \]

We compute
\begin{align*}
\frac{\partial H^2}{\partial t}&=2\frac{\Delta H}{H}-4\frac{|\nabla H|^2}{H^2}-2|A|^2+4\\
                               &=2(\Delta \log H)-2\frac{|\nabla H|^2}{H^2}-2(H^2-4)+4K.
\end{align*}
Taking into account of $\frac{\partial}{\partial t} d\mu_t =d\mu_t$ and integrating by parts, we obtain
\begin{align*}
\frac{d}{dt}\left(1-\frac{1}{16\pi}\int_{\Sigma_t} (H^2-4)d\mu_t\right)&=-\left(1-\frac{1}{16\pi}\int_{\Sigma_t} (H^2-4)d\mu_t\right)+\frac{1}{8\pi}\int_{\Sigma_t}\frac{|\nabla H|^2}{H^2}d\mu_t.
\end{align*} 

In view of \eqref{traceless},
\begin{align*}
\frac{d}{dt} \int_{\Sigma_t} |\mathring{A}|^2d\mu_t &=-\int_{\Sigma_t} |\mathring{A}|^2d\mu_t -\int_{\Sigma_t}\frac{|\nabla H|^2}{H^2}d\mu_t.
\end{align*}

We thus obtain the following proposition: 
\begin{proposition} Along an inverse mean curvature flow $\Sigma_t$, the modified Hawking mass $\tilde{m}(\Sigma_t)$ satisfies the following evolution equation:
\[\frac{d}{dt} \tilde{m}(\Sigma_t) = |\Sigma_t| \int_{\Sigma_t}\frac{|\nabla H|^2}{H^2}d\mu_t.\]
\end{proposition}

\section{The construction}

  We consider a family of star-shaped surfaces $\tilde{\Sigma}_s \subset \mathbb{H}^3$ that are described in the $(r, \theta)$ coordinates by
\begin{equation}\label{family_surface} \tilde{\Sigma}_s=\{ (\tilde{r}(\theta, s ),\theta):\ \theta\in S^{2} \} \end{equation}
for a smooth function $\tilde{r}(\theta, s)$ on $S^{2}\times \mathbb{R}^+$.  We assume $\tilde{r}(\theta, s)$ is of the following asymptotics as $s\rightarrow \infty$:
\begin{equation}\label{asymptotic} \tilde{r}(\theta, s)=c s+ f(\theta)+o(1)\end{equation} where $c$ is a constant and $f$ is a function on $S^2$.

Consider a new function
\[ \varphi(\theta, s):=-\int_{\tilde{r}(\theta, s)}^\infty\frac{dx}{\sinh(x)}\] such that $D_i \varphi=(\sinh^{-1} \tilde{r}) D_i\tilde{r}$.  
Let $\varphi_i=D_i\varphi$ and $\varphi_{ij}=D_j D_i\varphi$ be  covariant derivatives with respect to the round metric $\sigma$. Let $|\cdot |_{\sigma}$ be the norm corresponding to $\sigma$.  Let
 \[g_{ij}=\sinh^2 r(\sigma_{ij}+\varphi_i\varphi_j)\]
  be the induced metric of $\tilde{\Sigma}_s$ and $h_{ij}$ be the second fundamental form of $\tilde{\Sigma}_s$. 
Denote $v=\sqrt{1+|D\varphi|_{\sigma}^2}$ and the area form of $g_{ij}$  is \[\sqrt{\det g}=\sinh^2\tilde{r} \sqrt{1+|D\varphi|_{\sigma}^2}\sqrt{ \det \sigma}=(\sinh^2\tilde{r}) v \sqrt{ \det \sigma}. \]

$h_j^i=g^{ik}h_{kj}$ can be expressed in terms of $\varphi$:
\begin{equation}\label{2nd_fund} h_j^i=\frac{\cosh \tilde{r}}{v\sinh \tilde{r} }\delta_j^i-\frac{1}{v\sinh \tilde{r} }(\sigma^{ik}-\frac{\varphi^i\varphi^k}{v^2})\varphi_{kj}. \end{equation}

\begin{proposition} \label{asymptotics}Let $\tilde{\Sigma}_s$ be a family of surfaces in $\mathbb{H}^3$ that are radial graphs of the function $\tilde{r}(\theta, s)=cs+f(\theta)+o(1)$ and $g_{ij}$ be the induced metric on $\tilde{\Sigma}_s$. Then the limit of the rescaled metric $e^{-2cs} g_{ij}$ as $s\rightarrow \infty$ is round if and only if  $\lim_{s\rightarrow \infty} \tilde{m}(\tilde{\Sigma}_s)=0$. 

\end{proposition}

\begin{proof}
From \eqref{2nd_fund}, we compute the traceless part $|\mathring{A}|^2=|A|^2-\frac{1}{2}H^2$: 
\begin{align*}
&|\mathring{A}|^2=\frac{1}{v^2\sinh^2\tilde{r}}\left( \tilde{\sigma}^{ik}\varphi_{kj}\tilde{\sigma}^{jl}\varphi_{li}-\frac{1}{2}(\tilde{\sigma}^{ij}\varphi_{ij})^2 \right),
\end{align*} where $\tilde{\sigma}^{ik}=\sigma^{ik}-\frac{\varphi^i\varphi^k}{v^2}$ is the inverse of $\sigma_{ij}+\varphi_i\varphi_j$.

It is more convenient to express all terms in $\tilde{r}$ now. We have 
\[v^2=1+(\sinh^{-2} \tilde{r}) |D \tilde{r}|^2_\sigma=1+O(e^{-2cs})\] and
\[\sigma^{ik}-\frac{\varphi^i\varphi^k}{v^2}=\sigma^{ik}-\frac{\tilde{r}^i\tilde{r}^k}{ \sinh^2 \tilde{r}+|D \tilde{r}|^2_\sigma}=\sigma^{ik}+O(e^{-2cs}).\]
On the other hand, we compute \[\varphi_{ij}=(\sinh^{-1} \tilde{r}) B_{ij}\] where
\[B_{ij}=\tilde{r}_{ij}-\frac{\cosh \tilde{r}}{\sinh \tilde{r} }\tilde{r}_i\tilde{r}_j =f_{ij}-f_if_j+o(1)\] as $s\rightarrow \infty$. 

Therefore,
\begin{align*}
&\lim_{s\to\infty}|\tilde{\Sigma}_s|\int_{\tilde{\Sigma}_s}|\mathring{A}|^2 d\mu_s\\
&=\lim_{s\to\infty}\left(\int_{S^2}(\sinh^2 \tilde{r}) d\mu_{\sigma}\right)\left( \int_{S^2} (\sinh^{-2} \tilde{r}) |\mathring{B}_{ij}|_{\sigma}^2d\mu_{\sigma}  \right)\\
&=\lim_{s\to\infty} (\int_{S^2}\frac{e^{2f}}{4}d\mu_{\sigma})(\int_{S^2}\frac{4}{e^{2f}}|\mathring{B}_{ij}|^2_{\sigma}d\mu_{\sigma})\\
&=\int_{S^2}e^{2f} d\mu_{\sigma}\int_{S^2}|\mathring{D}^2 e^{-f}|^2_{\sigma}d\mu_{\sigma}.
\end{align*}
In particular, $\lim_{s\rightarrow \infty} \tilde{m}(\tilde{\Sigma}_s)\leq 0$ and the equality holds if and only if $e^{-f}$ is a linear combination of constants and first eigenfunctions of $S^2$. On the other hand, the limit of the rescaled induced metric is 
\[\lim_{s\to \infty} e^{-2cs} g_{ij}=\lim_{s\to \infty} e^{-2cs} (\sinh^2\tilde{r}) \sigma_{ij}=e^{2f} \sigma_{ij}\] 
In the view of the following lemma, the proof is complete.
\end{proof}

\begin{lemma} Let $\sigma$ be the standard round metric on $S^{n-1}$.
The conformal metric $e^{2f}\sigma$ has constant section curvature if and only if $e^{-f}$ is a linear combination of constants and first eigenfunctions of $S^{n-1}$.
\end{lemma}
\begin{proof}\label{conf_factor}
From \cite[Chapter 6 Lemma 1.2]{Schoen-Yau} the conformal group of $S^{n-1}$ is $SO(n,1)$. By direct computation, for each $\phi$ in the conformal group we have
\[ \phi^*\sigma=e^{2u}\sigma,\ \ e^{-u}=a_0+\sum_{i=1}^n a_iX^i \]
where $a_0,\ \dots ,\ a_n$ are some constants and $X^i, i=1\cdots n$ are the first $n$ eigenfunctions of $S^{n-1}$. On the other hand, if $e^{2f}\sigma$ has constant sectional curvature, it can be realized as $\phi^*\sigma$ for a conformal diffeomorphism $\phi$ \cite[Theorem 6.1.2]{Petersen}. The assertion follows.
\end{proof}

\subsection{Proof of Theorem 1}
Pick a function $\bar{f}$ on $S^2$ such that 
\[\int_{S^2}e^{2\bar{f}}d\mu_{\sigma}\int_{S^2}|\mathring{D}^2 e^{-\bar{f}}|^2_{\sigma}d\mu_{\sigma}= c_0> 0\] where $\mathring{D}^2 e^{-\bar{f}}$ is the traceless
part of the Hessian of $e^{-\bar{f}}$. 
Let $\tilde{\Sigma}_s$ be the family of surface described by $r=s+\bar{f}(\theta)$ in the $(r, \theta)$ coordinates. From the analysis in the last section, we have
\begin{equation}\label{limit}\lim_{s\to\infty}\tilde{m}(\tilde{\Sigma}_s)=-c_0.\end{equation}

Consider the inverse mean curvature flow $\Sigma_t^s$ with $\tilde{\Sigma}_s$ as the initial surface, where $t$ is the flow parameter.
Neves \cite[Theorem 3.1]{Neves} proved that for $s$ large enough, the following estimates hold for $\Sigma_t^s$:
\begin{equation}\label{pinching}\begin{split} |\tilde{\Sigma}_s||H^2-4|&\leq Ce^{-t}\\
 |\tilde{\Sigma}_s|^{3}|\nabla A|^2&\leq Ce^{-3t} \end{split}\end{equation}
where the constant $C$ does not depends on $s$.
Therefore, 
\begin{equation}\label{monotone} \frac{d}{dt}\tilde{m}(\Sigma^s_t)=|\Sigma^s_t|\int_{\Sigma^s_t}\frac{|\nabla H|^2}{H^2}d\mu^s_t \leq C|\tilde{\Sigma}_s|^{-1}e^{-t},\end{equation} where  $|\Sigma_t^s|=
|\tilde{\Sigma}_s| e^t$ is used. 
Pick $s_0$ large enough such that:

(1) $\tilde{\Sigma}_{s_0}$ is mean-convex.

(2) $\tilde{m}(\tilde{\Sigma}_{s_0})<-\frac{c_0}{2}$, which is possible by \eqref{limit}. 

(3) $C|{\tilde{\Sigma}}_{s_0}|^{-1}\leq \frac{c_0}{4}$.

Let $\Sigma_t$ be the inverse mean curvature flow with $\tilde{\Sigma}_{s_0}$ as the initial surface, \eqref{monotone} implies  \[\lim_{t\rightarrow \infty} \tilde{m}(\Sigma_t)< -\frac{c_0}{4}<0.\]

On the other hand, by the analysis in \cite{Ding, Gerhardt2} we know that starting from a mean-convex star-shaped surface, the solution of inverse mean curvature flow $\Sigma_t$ exists for all time and the surface $\Sigma_t$ is given as the graph of  \[\tilde{r}(\theta ,t)=\frac{t}{2}+f(\theta) +o(1)\] for a smooth function $f(\theta)$ on $S^2$ as $t\rightarrow \infty$. We can apply Proposition \ref{asymptotics} to conclude that 
$\Sigma_t$ does not converge to a round sphere after scaling. 

\qed

\subsection{Inverse mean curvature flows in the ball model}
We consider the ball model of the hyperbolic metric 
\[\frac{1}{(1-\frac{1}{4} \rho^2)^2}(d\rho^2+\rho^2 \sigma)\] which can be turned into the form 
\[ dr^2+(\sinh^2r) \sigma\] by the change of variable \[\rho=2-\frac{4}{e^r+1}.\]

Therefore, a solution of the inverse mean curvature flow defined by the radial function $r=\tilde{r}(\theta, t)=\frac{t}{2}+f(\theta)+o(1)$ in the $(r, \theta)$ coordinates
is the same the family of surfaced defined by $\rho=u(\theta, t)$ in the $(\rho, \theta)$ coordinates, where 
\[u(\theta, t)=2-\frac{4}{e^{\tilde{r}(\theta, t)}+1}.\]
In particular,
\[\lim_{t\rightarrow \infty} (u-2) e^{\frac{t}{2}}=-4 e^{-f(\theta)}.\] Our result indicates that $f$ does not have to be a linear combination of constants and 
first eigenfunctions of $S^2$. This should be compared with the claim in  \cite[Theorem 1.2 and 6.11]{Gerhardt2}.

\section{The higher dimensional case}
In this section, we show that the same conclusion holds in higher dimensions, i.e. there exists a star-shaped mean convex hypersurface $\Sigma_0$ in the $n$ dimensional hyperbolic space $\mathbb{H}^n$ for $n\geq 4$ such that the inverse mean curvature flow with $\Sigma_0$ as the initial data does not converge to a round sphere (of constant sectional curvature) after scaling.
Let $(\mathbb{H}^n,\bar{g})$ be the hyperbolic space with the metric \[\bar{g}=dr^2+(\sinh^2 r)\sigma,\] where $\sigma$ is the standard metric on $S^{n-1}$. 
For a closed hypersurface $\Sigma$ in $\mathbb{H}^n$, we consider the quantity:
\[ Q(\Sigma):=|\Sigma|^{-\frac{n-5}{n-1}}\int_{\Sigma}|\mathring{A}|^2 d\mu,\] where $\mathring{A}$ is the traceless part of the second fundamental form $A$. We have the following proposition analogous to Proposition  \ref{asymptotics}
\begin{proposition}\label{asymptotics_n}
Let $\tilde{\Sigma}_s$ be a family of hypersurfaces in $\mathbb{H}^n$ that are radial graphs of the functions $\tilde{r}(\theta, s)=cs+f(\theta)+o(1)$. Then
\begin{align*}
&\lim_{s\to\infty}Q({\tilde{\Sigma}}_s)=(\int_{S^{n-1}}e^{(n-1)f}\mu_{\sigma})^{-\frac{n-5}{n-1}}\int_{S^{n-1}}e^{(n-3)f}|\mathring{D^2} e^{-f}|^2_{\sigma}\mu_{\sigma}
\end{align*}
The limit of the rescaled induced metric is $e^{2f}\sigma$. It is a round metric if and only if $\lim_{s\to\infty} Q({\tilde{\Sigma}_s})=0$.
\end{proposition}
Now we turn to the inverse mean curvature flow. We need the following lemma concerning the evolution of $|\mathring{A}|^2$. 
\begin{lemma} \label{evolution_of_A}The following equation holds along the inverse mean curvature flow of a hypersurface in $\mathbb{H}^n$:
\[
\frac{\partial |\mathring{A}|^2}{\partial t}=2\nabla^j\left(\frac{\nabla_iH}{H^2} \right)\mathring{h}^i_j-\frac{4}
{n-1}|\mathring{A}|^2-2\frac{\mathring{h}^k_i\mathring{h}^j_k\mathring{h}^i_j}{H},
\] where $\mathring{h}^i_j=h^i_j-\frac{1}{n-1}H\delta^i_j$.
\end{lemma}
\begin{proof} We compute
\begin{align*}
\frac{\partial h_i^j}{\partial t}
&= \frac{1}{H^2} \, \nabla^j \nabla_i H - 2 \, \frac{\nabla_i H \, \nabla^j H}{H^3} - \frac{h_i^kh_k^j}{H} - \frac{1}{H} \, g^{mj} \, R_{\nu i\nu m} \\ 
&=\nabla^j\left(\frac{\nabla_iH}{H^2} \right)-\frac{1}{H}\left(\mathring{h}^k_i+\frac{H}{n-1}\delta^k_i\right)\left(\mathring{h}^j_k+\frac{H}{n-1}\delta^j_k\right)-\frac{1}{H}g^{mj}(-g_{im})\\
&=\nabla^j\left(\frac{\nabla_iH}{H^2} \right)-\frac{2}{n-1}\mathring{h}^j_i-\frac{\mathring{h}^k_i\mathring{h}^j_k}{H}+\left(\frac{1}{H}-\frac{H}{(n-1)^2}\right)\delta_i^j.
\end{align*}
Therefore,
\begin{align*}
\frac{\partial \mathring{h}_i^j}{\partial t}=\nabla^j\left(\frac{\nabla_iH}{H^2} \right)-\frac{2}{n-1}\mathring{h}^j_i-\frac{\mathring{h}^k_i\mathring{h}^j_k}{H}+\left(\frac{1}{H}-\frac{H}{(n-1)^2}-\frac{1}{n-1}\frac{\partial H}{\partial t}\right)\delta_i^j.
\end{align*}
Contracting the above equation with $\mathring{h}^i_j$ gives the desired formula.\end{proof}

 We know that starting from a mean-convex star-shaped surface, the solution of inverse mean curvature flow exists for all time and $\tilde{r}(.,t)-\frac{t}{n}$ converges to a smooth function $f$. 

The construction is similar to the $n=3$ case. Pick a function $\bar{f}$ on $S^{n-1}$ such that 
\[\left(\int_{S^{n-1}}e^{(n-1)\bar{f}}\mu_{\sigma}\right)^{-\frac{n-5}{n-1}}\int_{S^{n-1}}e^{(n-3)\bar{f}}|\mathring{D^2} e^{-\bar{f}}|^2_{\sigma}\mu_{\sigma}=c_0> 0\]
Consider $\tilde{\Sigma}_s=\{(s+\bar{f}(\theta),\theta)\}$ and let $\Sigma^s_t$ be the solution of the inverse mean curvature flow with initial data $\tilde{\Sigma}_s$. The pinching estimate \eqref{pinching} can be generalized to higher dimensions, and for  $s$ large enough, we have the following estimate on $\Sigma^s_t$:
\begin{equation}\label{pinching2}\begin{split} |\tilde{\Sigma}_s|^{\frac{4}{n-1}}\left(|H-(n-1)|^2+|\mathring{A}|^2\right)&\leq Ce^{-\frac{4t}{n-1}}\\
 |\tilde{\Sigma}_s|^{\frac{6}{n-1}}|\nabla A|^2&\leq Ce^{-\frac{6t}{n-1}}\end{split}, \end{equation}
where $C$ is a constant independent of $s$.

With this pinching estimate and Lemma \ref{evolution_of_A}, we deduce:
\begin{align*}
\frac{d}{dt}\int_{\Sigma^s_t}|\mathring{A}|^2 d\mu^s_t&=\frac{n-5}{n-1}\int_{\Sigma^s_t}|\mathring{A}|^2 d\mu^s_t+\int_{\Sigma^s_t}\left(-2\frac{\nabla_iH}{H^2} \nabla^j\mathring{h}^i_j-2\frac{\mathring{h}^k_i\mathring{h}^j_k\mathring{h}^i_j}{H} \right)d\mu^s_t\\
&\geq\frac{n-5}{n-1}\int_{\Sigma^s_t}|\mathring{A}|^2 d\mu^s_t-\tilde{C}|\tilde{\Sigma}_s|^{\frac{n-7}{n-1}}e^{\frac{n-7}{n-5}t},
\end{align*} where $\tilde{C}$ is independent of $s$. 
Thus $\frac{d}{dt}Q(\Sigma^s_t)\geq -\tilde{C}|\tilde{\Sigma}_s|^{\frac{-2}{n-1}}e^{\frac{-2t}{n-1}}$.

Pick $s_0$ large enough such that:

(1) $\tilde{\Sigma}_{s_0}$ is mean-convex.

(2) $Q(\tilde{\Sigma}_{s_0}) > c_0/2$.

(3) $-\tilde{C}|\tilde{\Sigma}_{s_0}|^{\frac{-2}{n-1}}>-\frac{1}{2(n-1)}c_0$. 

(4) The pinching estimate \eqref{pinching2} holds on $\Sigma^{s_0}_t$. 

Denote by $\Sigma_t$ the inverse mean curvature flow with $\tilde{\Sigma}_{s_0}$ as the initial data. The flow exists for all time and for $t$ large, $\Sigma_t$ is given as the graph of $\frac{t}{n-1}+f(\theta)+o(1)$ for another smooth function $f$ on $S^{n-1}$.

It follows from the above conditions on $s_0$ that $Q(\Sigma_0)>\frac{c_0}{2}$ and $\frac{d}{dt}Q(\Sigma_t)\geq -\frac{1}{2(n-1)}e^{\frac{-2t}{n-1}}c_0$. Therefore $\lim_{t\rightarrow \infty} Q(\Sigma_t)> c_0/4>0$. From Proposition \ref{asymptotics_n} the limit of the rescaled induced metric on $\Sigma_t$ is not a round one.

\end{document}